\documentclass[12pt]{amsart}

\usepackage{esint}
\usepackage{latexsym, amssymb, amsmath}
\usepackage{graphicx, color}
\usepackage[all,cmtip]{xy}

\allowdisplaybreaks[4]

\newtheorem{thm}{Theorem}[section]

\newtheorem{lem}[thm]{Lemma}
\newtheorem{prop}[thm]{Proposition}
\theoremstyle{definition}

\theoremstyle{remark}

\numberwithin{equation}{section}

\newcommand{\partialt}{ \frac{\partial}{\partial t}}
\newcommand{\D} {{\bf{D}}}

\newcommand{\R} {\mathcal{R}}
\newcommand{\V} {\mathcal{V}}

\newcommand{\n}{\nabla}
\newcommand{\End}{{\rm End}}
\newcommand{\Spin}{{\rm Spin}}
\newcommand{\SO}{{\rm SO}}
\newcommand{\Aut}{{\rm Aut}}
\newcommand{\Ad}{{\rm Ad}}
\newcommand{\GL}{{\rm GL}}
\newcommand{\Vol}{{\rm Vol}}

\begin{document}
\title[regularity for the spinor flow]{Regularity estimates for the gradient flow of a spinorial energy functional}

\author{Fei He}
\author{Changliang Wang}
\email{Fei He, Xiamen University,  hefei@xmu.edu.cn}
\email{Changliang Wang, Max Planck Institute for Mathematics, Bonn, cwangmath@outlook.com}

\footnote{F.H. was partially supported by the Fundamental
Research Funds for the Central Universities Grant No. 20720180007 and by NSFC11801474.}

\begin{abstract}
In this note, we establish certain regularity estimates for the spinor flow introduced and initially studied in \cite{AWW2016}. Consequently, we obtain that the norm of the second order covariant derivative of the spinor field becoming unbounded is the only obstruction for long-time existence of the spinor flow. This generalizes the blow up criteria obtained in \cite{Sc2018} for surfaces to general dimensions. As another application of the estimates, we also obtain a lower bound for the existence time in terms of the initial data. Our estimates are based on an observation that, up to pulling back by a one-parameter family of diffeomorphisms, the metric part of the spinor flow is equivalent to a modified Ricci flow.
\end{abstract}

\maketitle


\bigskip
\setcounter{section}{0}


\section{Introduction}
\noindent In \cite{AWW2016}, in order to study certain interesting special spinors, in particular parallel and Killing spinors, and related geometric structures, for example special holonomy metrics, from a variational point of view, Ammann, Weiss and Witt introduced the {\em spinorial energy functional} defined as
\begin{equation}\label{SpinorFunctional}
\begin{aligned}
\mathcal{E}: \mathcal{N} &\rightarrow \mathbb{R}_{\geq0},\\
 (g, \phi) & \mapsto \frac{1}{2}\int_M |\n^{g} \phi|^2_g dv_g,
\end{aligned}
\end{equation}
where $\mathcal{N}$ is the union of pairs $(g, \phi)$ of a Riemannian metric $g$ and a $g$-spinor of constant length one $\phi\in \Gamma(\Sigma_{g}M)$ over a closed spin manifold $M$.

In (\ref{SpinorFunctional}), $\n^{g}$ denotes the connection on the spinor bundle $\Sigma_{g}M$ induced by the Levi-Civita connection of $g$, and $|\cdot|_{g}$ the pointwise norm on $T^{*}M\otimes\Sigma_{g}M$ and $dv_{g}$ the volume form induced by $g$. For the simplicity of notations, in the following, we will omit the superscript and subscript $g$ in $\nabla^{g}$ and $|\cdot|_{g}$ once no ambiguity is caused.

Note that to compare spinor bundles induced by two different metric spin structures and then to calculate the variation of (\ref{SpinorFunctional}), one needs a connection on the Fr\'echet vector bundle $\mathcal{N}\rightarrow \mathcal{M}$, where $\mathcal{M}$ is the space of Riemnnian metrics on $M$. A natural one is the Bourguignon-Gauduchon (partial) connection introduced in \cite{BG1992} (also see \cite{AWW2016} for details). The evolution of the spinor $\phi$ is taken to be the vertical part with respect to this connection. In \cite{Wan91}, Wang studied the deformation of parallel spinors under variation of metrics in a different manner, where the spinor bundle is fixed and the variation of metrics reflects as a variation of spinor connections.

The variation formula of the functional (\ref{SpinorFunctional}) has been derived in \cite{AWW2016}. Consequently, in dimension of $M\geq 3$, they obtained that the critical points of (\ref{SpinorFunctional}) are metrics with parallel spinors, and metrics with Killing spinors are certain critical points of (\ref{SpinorFunctional}) subject to the constraint of fixed volume. Moreover, the negative gradient flow of (\ref{SpinorFunctional}) called the {\em spinor flow} is given by the coupled system
\begin{equation}\label{spinorFlow}
\begin{cases}
\begin{split}
\partialt g_{jk} & = \frac{1}{4}T_{ij} + \frac{1}{2}\langle \n_j \phi, \n_k \phi\rangle - \frac{1}{4}|\n \phi|^2 g_{jk}, \\
\partialt \phi & = \Delta \phi + |\n \phi|^2 \phi,
\end{split}
\end{cases}
\end{equation}
where
\begin{equation}\label{eqn: T tensor}
T_{jk} = g^{li}\nabla_{l} \tilde{T}_{ijk}
\end{equation}
is the divergence over the first variable of the following 3-tensor
\begin{equation}\label{eqn: T tilde tensor}
\tilde{T}(X,Y,Z) = \frac{1}{2} \langle X \wedge Y \cdot \phi, \nabla_{Z} \phi\rangle + \frac{1}{2}\langle X \wedge Z \cdot \phi, \nabla_{Y} \phi\rangle.
\end{equation}
It has been proved in \cite{AWW2016} that (\ref{spinorFlow}) is a weakly parabolic system whose degeneracy is caused by spin-diffeomorphism invariance, and the short-time existence of the initial value problem can be shown by pulling it back to a stricly parabolic system.  Behaviour of this flow on surfaces, on Berger spheres and on homogeneous spaces have been studied in \cite{AWW2016(2)}, \cite{Wit2016} and \cite{FSW2018} respectively, stability of the flow has been studied in \cite{Sc2017}.  Moreover,  a variational approach of investigating special metrics on 7-manifolds has already been carried out in an earlier work \cite{WW12}.

To study the long time behaviour of this flow, a fundamental question is to find a criteria for finite time singularities occurring. In has been proved in \cite{Sc2018} that on a closed 2-dimensional surface, if $|\n^2 \phi|$ stays uniformly bounded, then no finite time singularity will occur. We show that this result is true in general dimensions.

Our approach is to pull-back a solution of (\ref{spinorFlow}) by a family of diffeomorphisms to an equivalent system which looks like a modified Ricci flow coupled with an evolving spinor (see (\ref{spinorRicciFlow}) below). Then it is sufficient to do estimates for (\ref{spinorRicciFlow}). We obtain the following Bernstein type estimates in $L^2$-integral form.
\begin{thm}\label{l2Estimates}
Let $(M^n, g(t), \phi(t))$ be a solution of the spinor flow. For any constants $K>0$, $a>0$, suppose $|Rm|\leq K$, $|\n^2 \phi|\leq K$ and $|\n \phi|^2 \leq K$ on $B_{g(0)}(p, \frac{r}{\sqrt{K}})\times [0,\frac{a}{K}]$, then for any integer $k\geq 0 $ there exists a constant $C(n,a,k,r)$ such that
\begin{equation}
\fint_{B_\frac{r}{2\sqrt{K}}} |\n^k Rm|^2 dv(t)\leq \frac{CK^{2}}{t^k},
\end{equation}
\begin{equation}
\fint_{B_\frac{r}{2\sqrt{K}}} |\n^{2+k} \phi|^2 dv(t)\leq \frac{CK^{2}}{t^k},
\end{equation}
for $t \in(0,\frac{a}{K}]$. Here, $B_\frac{r}{2\sqrt{K}}$ denotes the geodesic ball centered at $p$ with radius $\frac{r}{2\sqrt{K}}$ with respect to the initial metric $g(0)$, $dv(t)$ denotes the volume form associated to the Riemannian metric $g(t)$, and as usual $\fint$ denotes the average, e.g.
$$\fint_{B_\frac{r}{2\sqrt{K}}} |\n^k Rm|^2 dv(t):=\left({\rm Vol}_{g(t)}\left(B_\frac{r}{2\sqrt{K}}\right)\right)^{-1}\int_{B_\frac{r}{2\sqrt{K}}} |\n^k Rm|^2 dv(t).$$
\end{thm}
Pointwise estimates then follow from standard embedding theorems. In view of the similarity of (\ref{spinorRicciFlow}) to the Ricci flow, we can adapt a method of \cite{KMW2016} to control the norm of the Riemann tensor in terms of the initial data and the first two derivatives of the spinor, see Lemma \ref{estimateOfRm} below. Hence as our first application of the estimates in Theorem \ref{l2Estimates}, we have
\begin{thm}\label{blowUpCondition}
Let $(g(t), \phi(t))$ be a solution to (\ref{spinorFlow}) or (\ref{spinorRicciFlow}) on a closed manifold $M$ and on time interval $[0, T)$, for some $T<\infty$, if $\sup\limits_{M\times[0,T)}|\n^2 \phi| < \infty$, then the flow can be extended to a larger time interval.
\end{thm}
Note that in Theorem \ref{blowUpCondition} we do not need to assume uniform bounds for $|\n \phi|$ along the spinor flow, since it is controlled by $|\n^{2}\phi|$ by Lemma \ref{lemma: second-derivative-control-frist} below.

As another application of Theorem \ref{l2Estimates}, we also obtain a lower estimate for the existence time in terms of initial data.
\begin{thm}\label{lowerBoundForExistenceTime}
Let $g(t), \phi(t)$ be a solution of (\ref{spinorRicciFlow}) on a closed manifold $M^n$ with $n\geq 3$. Suppose $\sup\limits_M |Rm|_{g(0)} \leq L K$ and $\sup\limits_M |\n^2 \phi|_{g(0)} \leq K$, there are constants $\Lambda$ and $\delta$ depending on $n$, $L$ and
\[ \fint_M K^{-2-i}|\n^i Rm|^2 dv(g(0)), \quad \fint_M K^{-2-i}|\n^{2+i} \phi|^2 dv(g(0)),\]
where $ i = 0,1,2,..., [\frac{n}{2}]+1$, such that the existence time interval contains $[0, \frac{\delta}{K}]$, and we have
\[
 \sup_{M \times [0, \frac{\delta}{K}]}|\n^2 \phi| \leq \Lambda K .
\]
\end{thm}

The rest of the note is organized as follows. In \S\ref{section: preliminaries on spin geometry}, we recall some basic facts in spin geometry that we may need in the note. In \S\ref{section: pull-back to Ricci flow}, we describe how to pull back the spinor flow system in (\ref{spinorFlow}) to an equivalent system in (\ref{spinorRicciFlow}), in which the metric evolving equation behaves similarly to the Ricci flow equation. Therefore, in \S\ref{section: Estimate of the Riemann tensor}, we are able to obtain $C^{0}$ estimate for Riemannian curvature tensor along the flow in (\ref{spinorRicciFlow}) by adapting the technique developed in \cite{KMW2016} for Ricci flow to this spinor flow. Furthermore, in \S\ref{section: l2 estimates of derivatives} we prove Theorem \ref{l2Estimates}. Then as applications of Theorem \ref{l2Estimates}, we prove Theorem \ref{blowUpCondition} in \S\ref{section: long time existence} and Theorem \ref{lowerBoundForExistenceTime} in \S\ref{section: lower bound estimate for the existence time}. Finally, in the appendix \S\ref{section: appendix} we recall some interpolation inequalities that we need in preceding sections.


\section{Preliminaries on spin geometry}\label{section: preliminaries on spin geometry}
\noindent In this section, we will briefly review some basic facts on spin geometry. Throughout this note, $M^{n}$ will be a connected closed spin manifold of dimension $n\geq2$, and $\mathcal{M}$ is the space of Riemannian metrics on $M$.

\subsection{Metric spin structure and spinor bundle}
For more details about this subsection, we refer to e.g. \cite{BFGK91}, \cite{Fri00}, and \cite{LM89}.

Fix a Riemannian metric $g\in\mathcal{M}$ on $M$, let $P_{\SO}(g)$ denote the orthonormal frame bundle with respect to $g$. A {\em metric spin structure} with respect to $g$ is a spin$(n)$-principal bundle $P_{\Spin}$ over $M$ together with an equivariant two-sheeted covering map
$$\xi: P_{\Spin}\rightarrow P_{\SO}(g),$$
which restricts to a non-trivial double covering map $\rho: \Spin(n)\rightarrow \SO(n)$ on each fiber. This is said to be equivalent to another metric spin structure $\xi^{\prime}: P^{\prime}_{\Spin}\rightarrow P_{\SO}(g)$ with respect to $g$, if there exists a principal $\Spin(n)$-bundle isomorphism $\chi: P_{\Spin}\rightarrow P^{\prime}_{\Spin}$ such that there is the commutative diagram
$$
\xymatrix{
P_{\Spin}  \ar[rr]^{\chi} \ar[rd]_{\xi}
  & &  P^{\prime}_{\Spin} \ar[ld]^{\xi^{\prime}}\\
  &P_{\SO}(g).
}
$$

The {\em complex} Clifford algebra $\mathbb{C}l_{n}$ of the standard Euclidean vector space $(\mathbb{R}^{n}, g_{\mathbb{R}^{n}})$ may be classified as
\begin{equation}\label{Clifford-algebra-classification}
\mathbb{C}l_{n}\cong
\begin{cases}
\End({\Sigma_{n}}), & \text{if} \quad n=even,\\
\End({\Sigma_{n}})\oplus\End({\Sigma_{n}}), & \text{if} \quad n=odd,
\end{cases}
\end{equation}
where ``$\cong$" means algebra isomorphism, and $\Sigma_{n}$ is a complex vector space of complex dimension $2^{[n/2]}$. The spin group $\Spin(n)$ is a subgroup of the group of invertible elements in $\mathbb{C}l_{n}$. Thus the inclusion $\Spin(n)\hookrightarrow\mathbb{C}l_{n}$ together with the algebra isomorphism in (\ref{Clifford-algebra-classification}) gives {\em the spin representation} of the spin group $\Spin(n)$ on $\Sigma_{n}$, and which is referred as $\mu: \Spin(n)\rightarrow \End(\Sigma_{n})$. Then {\em the spinor bundle} is defined to be the associated bundle
$$
\Sigma_{g}M:=P_{\Spin}\times_{\mu}\Sigma_{n}=(P_{\Spin}\times\Sigma_{n})/\sim,
$$
where $(\tilde{e}, \tilde{\phi})\sim(\tilde{e}\cdot g, \mu(g^{-1})\tilde{\phi})$ for $g\in\Spin(n)$. A section of the spinor bundle $\Sigma_{g}M$ is called a {\em spinor (field)}.

The (complex) {\em Clifford bundle} $\mathbb{C}l(g)$ is referred to be the union of complex clifford algebras of $(T_{x}M, g(x))$ for all $x\in M$. This bundle may be viewed as an associated bundle as
$$
\mathbb{C}l(g)=P_{\Spin}\times_{\Ad}\mathbb{C}l_{n},
$$
where $\Ad: \Spin(n)\rightarrow \Aut(\mathbb{C}l_{n})$ is given by $\Ad(g)(\varphi)=g\varphi g^{-1}$ for $g\in\Spin(n)$ and $\varphi\in\mathbb{C}l_{n}$. By this characterization of $\mathbb{C}l(g)$, one can easily check that the Clifford bundle $\mathbb{C}l(g)$ acts on the spinor bundle $\Sigma_{g}M$. In particular, the tangent bundle $TM\subset\mathbb{C}l(g)$ acts on $\Sigma_{g}M$, i.e. a vector field $X$ acts on a spinor $\phi$, written as $X\cdot\phi$ and this is called {\em Clifford multiplication}. In terms of a local representation $[\tilde{e}, \tilde{\phi}]$ of $\phi$, where for an open set $U\subset M$, $\tilde{\phi}: U\rightarrow \Sigma_{n}$, and $\tilde{e}: U\rightarrow P_{\Spin}$ covers a local orthonormal frame $e=(e_{1}, \cdots, e_{n}): U\rightarrow P_{\SO}(g)$, we have
\begin{equation*}
X\cdot\phi=\sum^{n}_{i=1}\langle X, e_{i}\rangle e_{i}\cdot\phi=\sum^{n}_{i=1}\langle X, e_{i}\rangle[\tilde{e}, E_{i}\cdot\tilde{\phi}],
\end{equation*}
where $\{E_{1}, \cdots, E_{n}\}$ is the standard basis of $\mathbb{R}^{n}$, and $E_{i}\cdot\tilde{\phi}$ is the Clifford multiplication induced by the algebra isomorphism in (\ref{Clifford-algebra-classification}).

The Clifford multiplication naturally extends to actions of tensor algebras. For example, for any $p$-form $\alpha$,
\begin{equation*}
\alpha\cdot\phi:=\sum_{1\leq i_{1}\leq \cdots \leq i_{q}\leq n}\alpha(e_{i_{1}}, \cdots, e_{i_{p}})e_{i_{1}}\cdot \dots \cdot e_{i_{p}}\cdot \phi.
\end{equation*}
In particular,
\begin{equation*}
(X\wedge Y)\cdot\phi=X\cdot Y\cdot\phi+g(X, Y)\phi.
\end{equation*}
Here we identify $TM$ and $T^{*}M$ by using the Riemannian metric $g$.

The Levi-Civita connection on the orthonormal frame bundle $P_{\SO}(g)$ naturally induces a connection on $P_{\Spin}$ and further on $\Sigma_{g}M$. In a local expression as above, the spinor covariant derivative of  $\phi$ in the direction of a vector field $X$ can be written as
\[
\nabla_X \phi = [\tilde{e}, X\tilde{\phi} + \frac{1}{4}\sum^{n}_{i,j=1} g(\nabla_X e_i, e_j) E_i \cdot E_j \cdot \tilde{\phi}].
\]
In other words,
\[ \nabla_X \phi =X(\phi) + \frac{1}{4} g(\nabla_X e_i ,e_j ) e_i \cdot e_j \cdot \phi.\]
Here and also in the following the Einstein summation convention is used.

Define the spinor curvature operator as
\[\R(X, Y) \phi := \nabla^{2}_{X, Y} \phi - \nabla^{2}_{Y, X} \phi=\nabla_X (\nabla_Y \phi) - \nabla_Y(\nabla_X \phi) - \nabla_{[X,Y]}\phi.\]
By direct calculation
\[ \R(e_i, e_j)\phi = \frac{1}{4}R_{ijkl}e_l \cdot e_k \cdot \phi.\]
Using the first Bianchi identity one can derive a formula for the Ricci curvature as
\[ e_i \cdot \R(e_j, e_i)\phi = -\frac{1}{2} R_{jk} e_k \cdot \phi. \]

Recall that there is an inner product on the spinor bundle $\Sigma_{g}M$ which is invariant under Clifford multiplication by unit vectors and compatible with the spinor covariant differentiation, i.e.
\[ \langle X\cdot \phi, X\cdot \psi\rangle = \langle \phi, \psi\rangle, \quad \text{if} \quad |X|=1, \]
\[
X \langle \phi, \psi\rangle = \langle \n_X \phi, \psi\rangle + \langle \phi, \n_X \psi \rangle .
\]
Using this inner product we get a formula for the Ricci tensor
\begin{equation}\label{eqn: Ricci-derivative-of-spinor}
R_{jk} = -2 \langle e_i \cdot \R(e_j, e_i) \phi, e_k \cdot \phi \rangle, \quad \text{if} \quad |\phi|=1.
\end{equation}
Hence the Ricci curvature tensor can be completely determined by the second covariant derivative of any unit spinor field, we will keep this fact in mind in the following of this note.

\subsection{Topological spin structure and universal spinor bundle} In order to compare spinors with respect to different Riemannian metrics, one needs to gather all spinor bundles $\Sigma_{g}M$ for all $g\in\mathcal{M}$ together into a single bundle, which will be the so called universal spinor bundle. For defining a universal spinor bundle, one needs a topological spin structure without referring to any specific Riemannian metric. For more details about these materials, see e.g. \cite{AWW2016}, \cite{BG1992}, and \cite{Swi93}.

Fix an orientation for the manifold $M^{n}$. Let $P_{\GL^{+}}$ be the oriented frame bundle over $M$, which is a principal $\GL(n)^{+}$-bundle over $M$. Then a {\em topological spin structure} is a principal $\widetilde{\GL}(n)^{+}$-bundle $P_{\widetilde{\GL}^{+}}$ together with an equivariant two-sheeted covering map
\begin{equation*}
\theta: P_{\widetilde{\GL}^{+}}\rightarrow P_{\GL^{+}},
\end{equation*}
which restricts to a non-trivial double covering map $\widetilde{\GL}(n)^{+}\rightarrow \GL(n)^{+}$ on each fiber. This is said to be equivalent to another topological spin structure $\theta^{\prime}: P^{\prime}_{\widetilde{\GL}^{+}}\rightarrow P_{\GL^{+}}$ if there exists an equivariant principal $\widetilde{\GL}(n)^{+}$-bundle isomorphism $\chi: P_{\widetilde{\GL}^{+}}\rightarrow P^{\prime}_{\widetilde{\GL}^{+}}$ such that there is the commutative diagram
$$
\xymatrix{
P_{\widetilde{\GL}^{+}} \ar[rr]^{\chi} \ar[rd]_{\theta}
     & &   P^{\prime}_{\widetilde{\GL}^{+}} \ar[ld]^{\theta^{\prime}}\\
     & P_{\GL^{+}}.
}
$$

It was shown in \cite{Swi93} that for any fixed metric $g$ the equivalent metric spin structures with respect to $g$ one-to-one correspond to equivalent topological spin structures.

The bundle of positive definite bilinear forms can be viewed as the associated bundle
\begin{equation*}
\odot^{2}_{+}T^{*}M=P_{\GL^{+}}\times_{p}(\GL(n)^{+}/\SO(n))=P_{\GL}/\SO(n),
\end{equation*}
where $p$ is the natural left action of $\GL(n)^{+}$ on the quotient space $\GL(n)^{+}/\SO(n)$. Now we choose and fix a topological spin structure $\theta: P_{\widetilde{\GL}^{+}}\rightarrow P_{\GL^{+}}$ throughout the rest of the notes. Then the bundle $\odot^{2}_{+}T^{*}M$ can also be viewed as
\begin{eqnarray*}
\odot^{2}_{+}T^{*}M
&=& P_{\widetilde{\GL}^{+}}\times_{\widetilde{p}}(\GL(n)^{+}/\SO(n))\cr
&=& P_{\widetilde{\GL}^{+}}\times_{\widetilde{p}^{\prime}}(\widetilde{\GL}(n)^{+}/\Spin(n))\cr
&=& P_{\widetilde{\GL}^{+}}/\Spin(n),
\end{eqnarray*}
where $\widetilde{p}$ and $\widetilde{p}^{\prime}$ are natural left actions of $\widetilde{\GL}(n)^{+}$ on $\GL(n)^{+}/\SO(n)$ and $\widetilde{\GL}(n)^{+}/\Spin(n)$, respectively.

Thus the projection $P_{\widetilde{\GL}^{+}}\rightarrow \odot^{2}_{+}T^{*}M$ is a principal $\Spin(n)$-bundle. The {\em universal spinor bundle} is then defined as the associated vector bundle
\begin{equation*}
\pi: \Sigma M:=P_{\widetilde{\GL}^{+}}\times_{\mu}\Sigma_{n}\rightarrow \odot^{2}_{+}T^{*}M.
\end{equation*}
By composing $\pi$ with the projection of the bundle projection $\odot^{2}_{+}T^{*}M\rightarrow M$, one can also view $\Sigma M$ as a fiber bundle over $M$ with fiber $(\widetilde{\GL}^{+}\times\Sigma_{n})/\Spin(n)$. Then a section $\Phi\in \Gamma(\Sigma M)$ of this bundle over $M$ determines a Riemannian metric $g_{\Phi}$ and a spinor $\phi_{\Phi}\in \Gamma(\Sigma_{g_{\Phi}}M)$ and vice versa. Therefore, we identify sections of the universal bundle $\Sigma M\rightarrow M$ with the corresponding pairs $(g, \phi)$.

The Bourguignon-Gauduchon (partial) connection (or horizontal distribution) in \cite{BG1992}, or equivalently generalized cylinder construction in \cite{BGM05}, produces the decomposition
\begin{equation}\label{eqn: BG-connection-decomposition}
T_{(g, \phi)}\Sigma M\cong \odot^{2}_{+}T^{*}_{x}M\oplus \Sigma_{g,x}M,
\end{equation}
where $(g, \phi)\in \Sigma M$ has base-point $x\in M$, and $\Sigma_{g, x}M$ is the fiber of the spinor bundle $\Sigma_{g}M$ with respect to the metric $g$ over $x\in M$. In the decomposition (\ref{eqn: BG-connection-decomposition}), the first factor is the horizontal part and the second factor the vertical part. For details of the decomposition (\ref{eqn: BG-connection-decomposition}), we also refer to \cite{AWW2016} and references therein.

The universal bundle of unit spinors $S(\Sigma M)$ is given by
\begin{equation*}
S(\Sigma M):=\{(g, \phi)\in \{g\}\times \Sigma_{g}M\subset \Sigma M \quad \mid \quad |\phi|_{g}=1\}\subset \Sigma M.
\end{equation*}
As in \cite{AWW2016}, let $\mathcal{F}$ and $\mathcal{N}$ denote respectively the spaces of smooth sections
\begin{equation*}
\mathcal{F}:=\Gamma(\Sigma M) \quad \text{and} \quad \mathcal{N}:=\Gamma(S(\Sigma M)).
\end{equation*}
They can been considered as Fr\'echet fiber bundles over $\mathcal{M}$. Then it follows from the decomposition (\ref{eqn: BG-connection-decomposition}) that
\begin{equation*}
T_{(g, \phi)}=\mathcal{H}_{(g, \phi)}\oplus T_{(g, \phi)}\mathcal{F}_{g}\cong \Gamma(\odot^{2}T^{*}M)\oplus \Gamma(\Sigma_{g}M),
\end{equation*}
and
\begin{equation*}
\begin{aligned}
T_{(g, \phi)}\mathcal{N}
&=\Gamma(\odot^{2}T^{*}M)\oplus \Gamma(\phi^{\perp})\\
&=\{(h, \varphi)\in \Gamma(\odot^{2}T^{*}M)\oplus \Gamma(\Sigma_{g}M) \mid \langle\varphi(x), \phi(x)\rangle=0, \forall x\in M\}.
\end{aligned}
\end{equation*}
Again the first factor in the decomposition is the horizontal part, and the second factor the vertical part. As mentioned in Introduction, in the spinor evolving equation in the spinor flow system, $\frac{\partial}{\partial t}\phi$ takes values in the vertical part.


\section{Pull-back to a modified Ricci flow}\label{section: pull-back to Ricci flow}
\noindent In this section, we will pull-back the spinor flow (\ref{spinorFlow}) to a modified Ricci flow coupled with an evolving spinor as in (\ref{spinorRicciFlow}). This helps us eliminate the term involving the second order derivative of spinor on the right hand side of the metric evolving equation in (\ref{spinorFlow}), and so that we can apply the technique developed in \cite{KMW2016} to do $C^{0}$-estimate for Riemann tensor in \S\ref{section: Estimate of the Riemann tensor}.

For simplicity we denote the covariant derivative in the direction of $e_i$ as $\n_i$. Using the fact that $|\phi| \equiv 1$ we have (recall tensors $\tilde{T}$ and $T$ defined in (\ref{eqn: T tilde tensor}) and (\ref{eqn: T tensor}), respectively,)
\[\tilde{T}_{ijk} = \frac{1}{2}\langle e_i \cdot \phi, e_j \cdot \nabla_{k} \phi + e_k \cdot \nabla_{j} \phi \rangle. \]
Hence for any vectors $X$ and $Y$, we can write $T$ as
\begin{eqnarray*}
T(X,Y)
& = & \frac{1}{2} \langle e_i \wedge X \cdot \nabla_{i} \phi, \nabla_{Y}\phi \rangle  + \frac{1}{2} \langle e_i \wedge Y \cdot \nabla_{i} \phi, \nabla_{X}\phi \rangle \\
&   & +\, \frac{1}{2} \langle e_i \wedge X \cdot \phi, \nabla _{i} \nabla_{Y}\phi \rangle+ \frac{1}{2} \langle e_i \wedge Y \cdot \phi, \nabla _{i} \nabla_{X}\phi \rangle.
 \end{eqnarray*}
Let ${\D}$ be the Dirac operator defined by
\[\D \phi := e_i \cdot \nabla_{i} \phi.\]
\begin{lem}\label{newFormulaForT}
For any tangent vectors $X$ and $Y$, we have
\begin{eqnarray*}
 T(X,Y)
 & = & -\, \frac{1}{2}Ric(X,Y) - 2\langle \nabla_X \phi, \nabla_Y \phi\rangle \\
 &   & +\, \frac{1}{2}\langle \D \phi, X \cdot \nabla_{Y}\phi + Y \cdot \nabla_{X}\phi \rangle\\
 &   & +\, \frac{1}{2}\langle \phi, X\cdot \nabla_Y \D \phi + Y \cdot \nabla_X \D \phi\rangle.
\end{eqnarray*}
\end{lem}
\begin{proof}
For simplicity we use the standard trick to choose an orthonormal frame $e_i$, $i=1,2,...,n$, and calculate at a point where $\n_i e_j = 0$.
First group $e_i \cdot$ and $\n_i$ whenever possible to introduce the Dirac operator to the expression of
\begin{eqnarray*}
T_{jk}
& = & \frac{1}{2}\langle \D \phi, e_j \cdot \nabla_{k}\phi + e_k \cdot \nabla_{j}\phi \rangle  - \langle \nabla_{j}\phi, \nabla_{k}\phi\rangle\\
&   & +\, \frac{1}{2} \langle \phi, e_j \cdot \D \nabla_{k} \phi +e_k \cdot \D \nabla_{j} \phi+  \nabla_{j}\nabla_{k}\phi + \nabla_{k} \nabla_{j}\phi \rangle. \\
\end{eqnarray*}
Differentiating $|\phi| \equiv 1$ twice yields
\[
\langle \n_j \phi, \n_k \phi\rangle = - \langle \phi, \n_j\n_k \phi\rangle.
\]
Hence
\begin{eqnarray*}
 T_{jk}
 & = & \frac{1}{2}\langle \D \phi, e_j \cdot \nabla_{k}\phi + e_k \cdot \nabla_{j}\phi \rangle + \frac{1}{2} \langle \phi,  e_j  \D \nabla_{k} \phi +e_k  \D \nabla_{j} \phi) \rangle \\
 &   & -\, 2 \langle \n_j \phi, \n_k \phi\rangle.
\end{eqnarray*}
Note that
\[
\D \nabla_k \phi = e_i \cdot \nabla_i \nabla_k \phi = e_i \cdot (\nabla_k\nabla_i \phi + \R(e_i, e_k )\phi) = \nabla_k \D \phi + \frac{1}{2}R_{kl} e_l \cdot \phi.
\]
We can rewrite the formula of $T$ as
\begin{eqnarray*}
 T_{jk}
 & = & \frac{1}{2}\langle \D \phi, e_j \cdot \nabla_{k}\phi + e_k \cdot \nabla_{j}\phi \rangle + \frac{1}{2}\langle \phi, e_j\nabla_k \D \phi + e_k \nabla_j \D \phi\rangle \\
&    & +\,  \frac{1}{2} \langle \phi, \frac{1}{2} R_{kl} e_j \cdot e_l \cdot \phi + \frac{1}{2} R_{jl} e_k \cdot e_l \cdot \phi \rangle - 2 \langle \nabla_j \phi, \nabla_k \phi \rangle \\
&  = & \frac{1}{2}\langle \D \phi, e_j \cdot \nabla_{k}\phi + e_k \cdot \nabla_{j}\phi \rangle + \frac{1}{2}\langle \phi, e_j\nabla_k \D \phi + e_k \nabla_j \D \phi\rangle \\
&    & -\, \frac{1}{2}R_{jk} - 2\langle \nabla_j \phi, \nabla_k \phi\rangle
\end{eqnarray*}

\end{proof}

Now define a vector field
\[X := \langle \phi, e_i \cdot \D \phi\rangle e_i.\]
Let $F(t)$ be the 1-parameter family of diffeomorphism generated by $X$, i.e.

\begin{align}
\begin{cases}
\frac{d}{dt}F(t) &= -\frac{1}{8} X,\\
F(0) &= id.
\end{cases}
\end{align}

$F(t)$ is a family of spin diffeomorphisms since they are isotopic to the identity.
Suppose $(g(t), \phi(t))$ is a solution to (\ref{spinorFlow}), pull-back $g(t),\phi(t)$ by $F(t)$, for simplicity we denote $F^*g$ and $F^*\phi$ still by $g$ and $\phi$. By the formula for the Lie derivative and the {\em metric Lie derivative} obtained in Proposition 17 in \cite{BG1992} (also see (8) in \cite{AWW2016})
\begin{eqnarray*}
L_X g_{ij}
& =  & -\, \langle e_i \cdot \nabla_j \phi + e_j \cdot \nabla_i \phi,   \D \phi\rangle\\
&    & +\, \langle \phi, e_i \cdot \nabla_j \D\phi+  e_j \cdot \nabla_i \D\phi\rangle,
\end{eqnarray*}

\begin{eqnarray*}
\tilde{L}_X \phi
& = & \nabla_X \phi - \frac{1}{4} dX^b \cdot \phi \\
&= & \langle \phi, e_i \cdot \D \phi\rangle \nabla_i \phi  - \frac{1}{4} (\langle \nabla_j \phi, e_i \cdot \D \phi\rangle\\
&  & +\, \langle \phi, e_i \cdot \nabla_j \D \phi\rangle ) e_j \wedge e_i \cdot \phi,
\end{eqnarray*}
we have
\begin{equation}\label{spinorRicciFlow}
\begin{cases}
\begin{split}
\partialt g_{jk} = & -\frac{1}{8}R_{jk} - \frac{1}{4}|\nabla \phi|^2 g_{jk}\\
&+  \frac{1}{4}\langle \D \phi, e_j \cdot \nabla_{e_k}\phi + e_k \cdot \nabla_{e_j}\phi \rangle, \\
\partialt \phi = & \Delta \phi + |\nabla \phi|^2 \phi -\frac{1}{8} \langle \phi, e_i \cdot \D \phi\rangle \nabla_i \phi \\
& + \frac{1}{32} (\langle \nabla_j \phi, e_i \cdot \D \phi\rangle  + \langle \phi, e_i \cdot \nabla_j \D \phi\rangle ) e_j \wedge e_i \cdot \phi.
\end{split}
\end{cases}
\end{equation}

Hence by pulling back the flow (\ref{spinorFlow}) by a 1-parameter family of diffeomorphisms we simplified the evolution equation of the Riemmanian metric, at the expense of complicating the evolution equation of the spinor field by introducing a new second order term
\[\frac{1}{32} \langle \phi, e_i \cdot \nabla_j \D \phi\rangle e_j \wedge e_i \cdot \phi. \]

Recall that under a smooth deformation of the metric $\partialt g = h$, we have
\[
\partialt R_{ijlk} = \frac{1}{2} (\nabla_i \nabla_k h_{jl} + \nabla_j \nabla_l h_{ik} - \nabla_i \nabla_l h_{jk} - \nabla_j \nabla_k h_{il} + R_{ijpk}h_{pl} + R_{ijlp}h_{kp}).
\]
Denote
\[
\V(h)_{ijlk}  := \nabla_i \nabla_k h_{jl} + \nabla_j \nabla_l h_{ik} - \nabla_i \nabla_l h_{jk} - \nabla_j \nabla_k h_{il}
\]
Under the flow (\ref{spinorRicciFlow}) we have
\begin{equation}\label{timeDerivativeOfRm}
\begin{split}
\partialt Rm = &-\frac{1}{16}\V(Ric) + \frac{1}{8} Rm(Ric) + Rm * \nabla \phi *\nabla \phi\\
 &+ \nabla^3 \phi * \nabla\phi + \nabla^2 \phi * \nabla^2 \phi .
\end{split}
\end{equation}
Here and in the rest of the notes, ``$*$" is allowed to involve three types of operations: contraction by the metric $g$, clifford multiplication by unit vectors, and contraction by the spinor inner product.

By the second Bianchi identity we can derive the heat-type evolution equation of the Riemann curvature tensor.
\begin{equation}\label{evolutionEquationOfRm}
\begin{split}
\partialt Rm =
&\frac{1}{16} \Delta Rm + Rm * Rm + Rm * \n \phi * \n \phi\\
& + \nabla^3 \phi * \nabla \phi+ \nabla^2 \phi * \nabla^2 \phi.
\end{split}
\end{equation}
After tracing we have the evolution equation of the Ricci tensor
\[
\partialt Ric = \frac{1}{16}\Delta Ric + Rm * Ric +  Rm * \nabla \phi *\nabla \phi + \nabla^3 \phi * \nabla\phi + \nabla^2 \phi * \nabla^2 \phi .
\]

For covariant derivatives of the Riemann tensor we have evolution equations
\begin{equation}\label{evolutionEquationForDerivativesOfRm}
\begin{split}
(\partialt - \frac{1}{16}\Delta)\n^k Rm
= &  \sum_{l=0}^k \sum_{r=0}^l \n^{1+r} \phi * \n^{1+l-r}\phi * \n^{k-l} Rm \\
&  +  \sum_{l=0}^k \n^l Rm * \n^{k-l} Rm
+\sum_{l=0}^{k+2}\n^{1+l}\phi * \n^{3+k-l} \phi ,
\end{split}
\end{equation}
where $\n^k$ denote the $k$-th covariant derivative.
For the spinor we have
\begin{eqnarray*}
(\partialt - \Delta) \nabla^k \phi
& = & \frac{1}{32} \langle \phi, e_p \cdot  \nabla^k \nabla_q \D \phi\rangle e_q \wedge e_p \cdot \phi \cr
&   & + \sum_{p=0}^k \sum_{q=0}^{k+1-p} \nabla^{1+p}\phi * \nabla^{q}\phi * \nabla^{k+1-p-q}\phi \cr
&   & + \sum_{0\leq l \leq k-1} \nabla^l Rm * \nabla^{k-l}\phi + \sum_{1\leq l \leq k } \nabla^l Ric * \nabla ^{k-l} \phi.
\end{eqnarray*}


\section{Estimates of the Riemann tensor}\label{section: Estimate of the Riemann tensor}
\noindent In this section we show that the first two derivative of $\phi$ controls the growth of $|Rm|$ along the flow (\ref{spinorRicciFlow}). Since this is trivial in dimension $\leq 3$, we can assume here the dimension is $n\geq 4$. We use the method of \cite{KMW2016}.
\begin{lem}\label{estimateOfRm}
Let $(g(t), \phi(t))$ be a solution of (\ref{spinorRicciFlow}) with $|\n \phi|^2, |\n^2 \phi| \leq K$ on $B(r)\times [0,T]$, where we can take $B(r)$ to be a geodesic ball with respect to $g(0)$ with radius $r$. Then we have
\[ \sup_{B(r/2)\times[0,T]}|Rm| \leq C(n,K, r^{-1}, T, \sup_{B(r)}|Rm|(g(0))).\]
\end{lem}
\begin{proof}
We use $C$ to denote constants that may depend on $n, K, r^{-1}$ and $p$.  For simplicity, we allow C to vary from term to term. Let $\eta$ be a cut-off function independent of $t$, and $|\n \eta|^2 \leq Cr^{-1} \eta$.  By (\ref{timeDerivativeOfRm}) we have
\begin{eqnarray*}
&   & \partialt \int \eta^p |Rm|^p dv(t) \cr
& = & \int p \eta^p |Rm|^{p-2} \langle Rm, \partialt Rm\rangle + \frac{1}{2}tr_g(\partialt g) \eta^p|Rm|^p  \cr
& \leq & C\int  \eta^{p-2}|\n \eta|^2|Rm|^{p-2} |\nabla Ric|^2 +  \eta^p|Rm|^{p-2} |\nabla Rm| |\nabla Ric|  \cr
&   & +\, C\int \eta^p|Rm|^p  + \eta^{p-1}|\n \eta| |Rm|^{p-1}+ \eta^p|Rm|^{p-2}|\nabla Rm| \cr
\end{eqnarray*}
where we used integration by part to deal with $\V(Ric)$ and $\nabla^3 \phi$. In this section we can allow the constants to depend on $p$.

By the heat type equations we have
\begin{eqnarray*}
\frac{1}{8}|\nabla Rm|^2
& \leq & (\frac{1}{16}\Delta - \partialt)|Rm|^2 + C(n)(|Rm|^3 + |Rm| |\nabla^2 \phi|^2) \cr
&      & +\, C |Rm|^2+ Rm * \nabla^3 \phi * \nabla \phi,
\end{eqnarray*}
and
\begin{eqnarray*}
\frac{1}{8}|\nabla Ric|^2
& \leq & (\frac{1}{16} \Delta - \partialt)|Ric|^2 + C(n) ( |Rm||Ric|^2 + |Rm| |Ric| |\nabla \phi|^2  ) \\
&      & +\, C(n) |Ric| |\nabla^2 \phi|^2 + Ric * \nabla^3 \phi * \nabla \phi .
\end{eqnarray*}
Note that $|\n^2 \phi| \leq K$ implies $|Ric| \leq nK$. We need to deal with the following (other terms are easy to handle)
\[
\begin{split}
& I =  \int  \eta^{p-1} |Rm|^{p-2} |\nabla Ric|^2, \\
& II =  \int \eta^p |Rm|^{p-2} |\nabla Rm| |\nabla Ric|, \\
& III =  \int \eta^p |Rm|^{p-2} |\nabla Rm| .
\end{split}
\]

\begin{eqnarray*}
I
& \leq & 8 \int \eta^{p-1}|Rm|^{p-2} (\frac{1}{16} \Delta - \partialt) |Ric|^2 + C \eta^{p-1}|Rm|^{p-1} \cr
&      & +\, C \int  \eta^{p-1}|Rm|^{p-2} +  \eta^{p-1} |Rm|^{p-2} Ric* \nabla^3 \phi * \nabla \phi   \cr
& \leq & -\, 8\int \eta^{p-1} |Rm|^{p-2}\partialt |Ric|^2  \cr
&    & +\, C \int  \eta^{p-1}|Rm|^{p-3} |\nabla Rm| |\nabla Ric|+ C \int \eta^{p-2}|\n \eta||Rm|^{p-2}|\n Ric| \cr
&    & +\, C\int \eta^{p-1 }|Rm|^{p-1} + \eta^{p-1}|Rm|^{p-2}+ \eta^{p-2}|\n \eta| |Rm|^{p-2} \cr
&    & +\, C \int \eta^{p-1}|Rm|^{p-3} |\nabla Rm| + C \int \eta^{p-1}|Rm|^{p-2} |\nabla Ric|.
\end{eqnarray*}
By the Cauchy inequality the above implies
\begin{eqnarray*}
I
& \leq & C \int \eta^{p-1}|Rm|^{p-4} |\nabla Rm|^2 + \eta^{p-1}|Rm|^{p-1} + \eta^{p-2}|Rm|^{p-2} \cr
&      & -\, 16 \partialt \int \eta^{p-1} |Rm|^{p-2}|Ric|^2 dv +  16\int \eta^{p-1} |Ric|^2\partialt |Rm|^{p-2}.
\end{eqnarray*}
The first term in the above is
\begin{eqnarray*}
&    & \int \eta^{p-1} |Rm|^{p-4}|\nabla Rm|^2 \cr
& \leq & 8 \int \eta^{p-1}|Rm|^{p-4} (\frac{1}{16}\Delta - \partialt )|Rm|^2 + C \int \eta^{p-1} (|Rm|^{p-1} + |Rm|^{p-2})  \cr
&      & +\, C \int \eta^{p-1}|Rm|^{p-3} + C \int\eta^{p-1} |Rm|^{p-4} Rm * \nabla^3 \phi * \nabla \phi,
\end{eqnarray*}
after integration by part we can derive
\begin{eqnarray*}
&      &\int \eta^{p-1} |Rm|^{p-4}|\nabla Rm|^2 \\
& \leq & -\, C\partialt \int \eta^{p-1} |Rm|^{p-2} \\
&      & +\, C \int \eta^{p-1}(|Rm|^{p-1} + |Rm|^{p-2}+ |Rm|^{p-3}+ |Rm|^{p-4}) \\
&      & +\, C \int \eta^{p-2} (|Rm|^{p-2}+ |Rm|^{p-3}).
\end{eqnarray*}
For the last term in the above estimate of $I$, we have
\begin{eqnarray*}
&      & \int \eta^{p-1} |Ric|^2 \partialt |Rm|^{p-2} \\
& \leq  &(p-2)\int \eta^{p-1} |Ric|^2 |Rm|^{p-4}\langle Rm, \partialt Rm \rangle  + C \int \eta^{p-1}|Rm|^{p-2}\\
& \leq & C(p-2) \int \eta^{p-1}|Ric|^2 |Rm|^{p-4}|\nabla Rm| (|\nabla Ric|+ K) \\
&      & +\, C(p-2)\int \eta^{p-1}|Ric| |Rm|^{p-3} |\nabla Ric| (|\nabla Ric | + K ) \\
&      & +\, C\int \eta^{p-2}|\n \eta| |Ric|^2 |Rm|^{p-3} ( |\n Ric| + K )\\
&      & +\, C(n, K)(p-2)\int \eta^{p-1}|Rm|^{p-3}+ C \int \eta^{p-1}|Rm|^{p-2} \\
&  \leq & C(1+ \frac{1}{\epsilon})\int\eta^{p-1} |Rm|^{p-4}|\nabla Rm|^2 + \eta^{p-2}( |Rm|^{p-2} +|Rm|^{p-3}+ |Rm|^{p-4})\\
&     & +\, \epsilon I,
\end{eqnarray*}
where we have used $|\nabla Ric| \leq (n-1) |\nabla Rm|$. Therefore
\begin{eqnarray*}
I
& \leq & -\, C \partialt \int \eta^{p-1} |Rm|^{p-2} -\, C\partialt \int \eta^{p-1}  |Rm|^{p-2}|Ric|^2 \\
&      & +\, C \int \eta^{p-1}  |Rm|^{p-1}+ \eta^{p-2} (|Rm |^{p-2} + |Rm|^{p-3} + |Rm|^{p-4}).
\end{eqnarray*}

Now we estimate the term
\begin{eqnarray*}
II
& = & \int\eta^p |Rm|^{p-2}|\nabla Rm | |\nabla Ric | \\
& \leq & C \int \eta^p |Rm|^{p-1}|\nabla Ric |^2 + C\int \eta^p |Rm|^{p-3}|\nabla Rm|^2 \\
& = & II_1 + II_2 .
\end{eqnarray*}

\begin{eqnarray*}
II_1
& \leq & C \int \eta^p |Rm|^{p-1} (\frac{1}{16}\Delta - \partialt)|Ric|^2 + \eta^p|Rm|^p+\eta^p |Rm|^{p-1} \\
&      & +\, \int \eta^p |Rm|^{p-1}Ric *\nabla^3 \phi * \nabla \phi \\
& \leq & C \int \eta^{p-1}(|\n \eta|+1)|Rm|^{p-1} |\n Ric| +\eta^p |Rm|^{p-2}|\nabla Rm| |\nabla Ric| \\
&      & +\, C \int \eta^p (|Rm|^p + |Rm|^{p-1})\\
&      & +\, C\int  \eta^{p-1}|\n \eta||Rm|^{p-1} + \eta^p |Rm|^{p-2}|\n Rm| \\
&      & -\, C \partialt \int \eta^p |Rm|^{p-1}|Ric|^2 + C \int \eta^p |Ric|^2 \partialt |Rm|^{p-1} \\
& \leq & \frac{1}{4}II_1 + C II_2 + C \int \eta^p |Rm|^p + \eta^{p-1}(|Rm|^{p-1} + |Rm|^{p-2} )\\
&      & -\, C \partialt \int \eta^p |Rm|^{p-1}|Ric|^2 + C \int \eta^p |Ric|^2 \partialt |Rm|^{p-1}. \\
\end{eqnarray*}
Similarly as above we derive
\begin{eqnarray*}
\int \eta^p |Ric|^2 \partialt |Rm|^{p-1}
& \leq & II_2 + C\int \eta^{p-1}(|Rm|^{p-1}+ |Rm|^{p-2} + |Rm|^{p-3}) \\
&      & +\, C\int \eta^p|Rm|^p + \frac{1}{4C}II_1,
\end{eqnarray*}
hence
\[
II_1 \leq  C II_2 - C\partialt \int  \eta^p |Rm|^{p-1}|Ric|^2 + C \int \eta^{p-1}(|Rm|^{p-3} + |Rm|^{p-1}). \\
\]

\begin{eqnarray*}
II_2
& \leq & C \int \eta^p (|Rm|^{p-3} (\frac{1}{16}\Delta - \partialt )|Rm|^2 + |Rm|^p + |Rm|^{p-2} )\\
&      & +\, C \int \eta^p |Rm|^{p-3} Rm * \nabla^3 \phi * \nabla \phi \\
&  \leq & -\, C \partialt \int \eta^p |Rm|^{p-1} \\
&      & +\, C \int \eta^p |Rm|^p + \eta^{p-1}(|Rm|^{p-1}+ |Rm|^{p-2} + |Rm|^{p-3})
\end{eqnarray*}
by similar arguments as before. Clearly

\[
III \leq C {II_{2}} + C \int \eta^p |Rm|^{p-1}.
\]
 Therefore we have shown
 \begin{eqnarray*}
 &    & \partialt \int \eta^p |Rm|^p dv \\
 & \leq & -\, C\partialt \int \eta^{p-1} |Rm|^{p-2}|Ric|^2 + C \partialt \int \eta^{p-1} |Rm|^{p-2} \\
 &      & -\, C \partialt \int \eta^p|Rm|^{p-1}|Ric|^2 + C \partialt \int \eta^{p} |Rm|^{p-1} \\
 &      & +\, C \int \eta^p |Rm|^p +  \eta^{p-4}|Rm|^{p-4}\\
 \end{eqnarray*}

Using Gronwall\rq{}s inequality, and use Young\rq{}s inequality to interpolate will yield an estimate of $\int |Rm|^p dv(t)$,
\[
\int \eta^p |Rm|^p dv(g(t)) \leq e^{Ct} \left( \int \eta^p |Rm| dv(g(0)) + CV_{g(0)}({\rm spt}(\eta))t \right),
\]
where the constant $C$ depends on $n, p, K $ and $\sup \eta^{-1}|\n \eta|^2$.

 Then the Nash-Moser iteration will give us pointwise estimate,
\begin{eqnarray*}
&    & \sup_{B(r/2)\times [0,T]} |Rm|\\
& \leq & C(n)e^{c(n,K)T}((C(n,p)r^2\sup_{[0,T]} |Rm|_{L^p(r)}(t))^{\alpha(n,k,p)} + r^{-\beta(n,k)}) \\
&      & \times \left( V(r)^{-1}|Rm|_{L^k(B(r)\times[0,T])}+ \sup_{B(r)} |Rm|(g(0))\right).
\end{eqnarray*}
Note the only non-standard step here is to use integral by part and absorption arguments to get rid of $\nabla^3 \phi$ . See for example \cite{Li2012} for details of Nash-Moser iteration and Theorem 1.2 of \cite{LT1991} where the initial value was taken into account.

\end{proof}


\section{$L^2$ estimates of derivatives}\label{section: l2 estimates of derivatives}
\noindent In this section we use the $L^2$ method to obtain derivative estimates for (\ref{spinorRicciFlow}). We assume $|\n^2 \phi|$, $|\n \phi|^2$, $|Rm| \leq K$ in the support of a cut-off function $\eta$, with $|\n \eta|^2 \leq L \eta$.

\begin{lem}\label{lemma: ddt of L2 integral of derivatives of Rm} For each integer $k\geq 0$ and $m\geq 2k$, there exists a constant $C$ depending on $n, k, m, K,L$, such that
\begin{eqnarray*}
&     &\partialt \int \eta^m |\n^k Rm|^2 \\
& \leq & -\, \frac{1}{128} \int \eta^m |\n^{k+1} Rm|^2  + C \int (\eta^m |\n^k Rm|^2 + \eta^{m}|\n^{k+2} \phi|^2 ) \\
&      & +\, C \int _{\eta>0} |Rm|^2 + |\n^2 \phi|^2 + |\n \phi|^2 + |\phi|^2.
\end{eqnarray*}
\end{lem}
\begin{proof}

\begin{eqnarray*}
&     & \partialt \int \eta^m |\n^k Rm|^2  \\
& \leq &  \int \eta^m \langle \n^k Rm, \frac{1}{16}\Delta \n^k Rm\rangle  + \int \eta^m \langle \n^k Rm,  \sum_{l=0}^k \n^l Rm * \n^{k-l} Rm \rangle \\
&     & +\, \int \eta^m \langle \n^k Rm, \sum_{l=0}^k \sum_{r=0}^l \n^{1+r}\phi * \n^{1+l-r} \phi * \n^{k-l}Rm \rangle\\
&     & +\, \int \eta^m \langle \n^k Rm, \sum_{l=0}^{k+2} \n^{1+l}\phi * \n^{3+k-l}\phi \rangle + C \int \eta^m |\n^k Rm|^2\\
&   = & I_1 + I_2 + I_3 + I_4 + C \int \eta^m |\n^k Rm|^2.
\end{eqnarray*}

\begin{eqnarray*}
I_1
& = &\int \eta^m \langle \n^k Rm, \frac{1}{16}\Delta \n^k Rm\rangle \\
& = & -\, \frac{1}{16}\int\eta^m |\n^{k+1}Rm|^2 + m \eta^{m-1} \langle \n \eta \otimes \n^k Rm, \n^{k+1} Rm \rangle  \\
& \leq & -\, \frac{1}{32} \int \eta^m |\n^{k+1} Rm|^2 + Cm^2 \int \eta^{m-1} |\n^k Rm|^2.  \\
\end{eqnarray*}
To handle the term $\int \eta^{m-1} |\n^k Rm|^2$, we use integration by parts to get
\begin{eqnarray*}
&    &\int \eta^{m-1}|\n^{k} Rm|^2 \\
&  = &  \int \eta^{m-1} \langle \n^{k-1} Rm, \n^{k+1} Rm\rangle  + \int  \langle \n \eta^{m-1} \otimes \n^{k-1} Rm, \n^k Rm\rangle \\
&  \leq & \epsilon \int \eta^m |\n^{k+1} Rm|^2 + \frac{1}{2} \int \eta^{m-1} |\n^k Rm|^2  + C\epsilon^{-1}\int \eta^{m-2}|\n^{k-1} Rm|^2,
\end{eqnarray*}
then iterate the above inequality to get
\[
\int \eta^{m-1}|\n^{k} Rm|^2 \leq \epsilon \int \eta^m |\n^{k+1} Rm|^2 + C \epsilon^{-1} \int_{\eta>0} |Rm|^2 .
\]
Then we can choose $\epsilon$ properly to get
\[
\begin{split}
I_1 \leq  - \frac{1}{64} \int \eta^m |\n^{k+1} Rm|^2 + C \int \eta^m |\n^k Rm|^2 + C \int _{\eta>0} |Rm|^2.
\end{split}
\]
Then we estimate
\begin{eqnarray*}
I_2
& = &  \int \eta^m \langle \n^k Rm,  \sum_{l=0}^k \n^l Rm * \n^{k-l} Rm \rangle \\
& \leq & C \left( \int \eta^m |\n^k Rm|^2 + \int_{\eta>0} |Rm|^2 \right),
\end{eqnarray*}

\begin{eqnarray*}
I_3
& = &  \int \eta^m \langle \n^k Rm, \sum_{l=0}^k \sum_{r=0}^l \n^{1+r}\phi * \n^{1+l-r} \phi * \n^{k-l}Rm \rangle\\
& \leq & C \left( \int \eta^m |\n^k Rm|^2 + \int_{\eta>0} |Rm|^2 \right) \\
&   & +\, C \left( \int \eta^m |\n^{1+k}\phi|^2 + \int_{\eta>0} |\n \phi|^2\right).
\end{eqnarray*}
Moreover, as before,
\begin{equation*}
\int \eta^{m} |\nabla^{1+l}\phi|^{2} \leq \int \eta^{m-1} |\nabla^{1+l}\phi|^{2} \leq C \int \eta^{m} |\nabla^{2+k}\phi|^{2} + C \int_{\eta>0} |\phi|^{2}.
\end{equation*}

\begin{eqnarray*}
I_4
& =  &\int \eta^m \langle \n^k Rm, \sum_{l=0}^{k+2} \n^{1+l}\phi * \n^{3+k-l}\phi \rangle \\
& = & \int \eta^m \langle \n^k Rm, \sum_{j=0}^{k} \n^{2+j}\phi * \n^{2+k-j}\phi \rangle + \int \eta^m \langle \n^k Rm, \n^{k+3} \phi * \n \phi \rangle  ,
\end{eqnarray*}
first integrate by part to lower the order of $\n^{3+k} \phi$, then use Cauchy inequality and the interpolation lemma in the appendix to get
\begin{eqnarray*}
I_4 & \leq & \frac{1}{128} \int \eta^m |\n^{k+1} Rm|^2 + C \left(\int \eta^m |\n^k Rm|^2 + \int_{\eta>0} |Rm|^2\right) \cr
& & +\, C\left( \int \eta^{m} |\n^{k+2} \phi|^2 + \int _{\eta>0 } |\n^2 \phi|^2 \right),
\end{eqnarray*}
where we used the same argument as in the estimate of $I_1$ to control the term $\int \eta^{m-1}|\n^k Rm|^2$.
By the above estimates we have the lemma.
\end{proof}

\begin{lem}\label{lemma: ddt of L2 integral of derivatives of phi}
For any integer $k\geq 1$, $m> 2k$, there exists a constant $C(n,k,m,K,L)$ such that
\begin{eqnarray*}
\partialt \int \eta^m |\n^k \phi|^2
& \leq &   - \frac{1}{4} \int \eta^m |\n^{k+1} \phi|^2 + C \int \eta^m |\n^k \phi|^2 \cr
&      & +\, \eta^m |\n^{k-1} Rm|^2  +  C \int _{\eta>0} |\n \phi|^2 + |\phi|^2+ |Rm|^2.
\end{eqnarray*}
\end{lem}
\begin{proof}

\begin{eqnarray*}
&   & \partialt  \int \eta^m | \nabla^k \phi |^2 \cr
& = & 2\int \eta^m \langle \n^k \phi, \Delta \n^k \phi \rangle +  \frac{1}{32} \eta^m \langle \phi, e_p \cdot  \nabla^k \nabla_q \D \phi\rangle \langle \n^k \phi,  e_q \wedge e_p \cdot \phi \rangle\cr
&   & + \sum_{p=0}^k \sum_{q=0}^{k+1-p} \int \eta^m \langle \n^k \phi, \nabla^{1+p}\phi * \nabla^{q}\phi * \nabla^{k+1-p-q}\phi \rangle \cr
&   & +\sum_{0\leq l \leq k-1} \int \eta^m \langle \n^k \phi, \nabla^l Rm * \nabla^{k-l}\phi \rangle  \cr
&   & + \sum_{1\leq l \leq k } \int \eta^m \langle \n^k \phi, \nabla^l Ric * \nabla ^{k-l} \phi \rangle \cr
&  = & II_1 + II_2 + II_3 + II_4 + II_5.
\end{eqnarray*}

\begin{eqnarray*}
II_1 & = & 2 \int \eta^m \langle \n^k \phi, \Delta \n^k \phi\rangle \cr
& =& -\, 2\int \eta^m |\n^{k+1} \phi|^2 + 2 \int \eta^{m-1}\langle \n \eta \otimes \n^k \phi, \n^{k+1} \phi\rangle \cr
& \leq &  -\int \eta^m |\n^{k+1} \phi|^2 + C \int \eta^{m-1} |\n^k \phi|^2\cr
& \leq & -\, \frac{1}{2}\int \eta^m |\n^{k+1} \phi|^2 +  C \int \eta^{m-2}|\n^{k-1} \phi|^2   \cr
& \leq & -\, \frac{1}{2}\int \eta^m |\n^{k+1} \phi|^2 + C \int \eta^m |\n^k \phi|^2 + \int _{\eta>0} |\phi|^2.\\
\end{eqnarray*}

\[
 II_2 = \frac{1}{32} \int  \eta^m \langle \phi, e_p \cdot  \nabla^k \nabla_q \D \phi\rangle \langle \n^k \phi,  e_q \wedge e_p \cdot \phi \rangle.
\]
By commuting covariant derivatives and the Dirac operator, we have
\[
\n^k \n_q \D \phi = \n_q \n^k \D \phi +\sum_{i=0}^{k-1} \n^i Rm * \n^{k-1-i}\D\phi.
\]
\begin{equation*}
\D \n^k \phi - \n^k \D \phi = \sum_{i=0}^{k-1} \n^i Ric * \n^{k-1-i}\phi.
\end{equation*}
Then integration by part yields
\begin{eqnarray*}
  II_2
&= &\frac{1}{32} \int \eta^m \langle \phi, e_p \cdot  \nabla_q \nabla^k  \D \phi\rangle \langle \n^k \phi,  e_q \cdot e_p \cdot \phi - \delta_{pq} \phi \rangle \cr
& &+ \int \eta^m  \sum_{i=0}^{k-1} \n^i Rm * \n^{k-1-i}\D\phi * \n^k \phi * \phi * \phi \cr
&= & -\, \frac{1}{32} \int \eta^m \langle \phi, e_p \cdot  \nabla^k  \D \phi\rangle (\langle - \D \n^k \phi,  e_p \cdot \phi\rangle  -\langle \n_q \n^k \phi,  \delta_{pq} \phi \rangle)\cr
& & +\, \frac{1}{32}\int m \eta^{m-1}\n_q \eta \langle \phi, e_p \cdot  \nabla^k\D \phi\rangle \langle \n^k \phi,  e_q \wedge e_p \cdot \phi \rangle \\
& & +\, \int \eta^m \n^{k}\D \phi * \n^k \phi * \n \phi * \phi \cr
& & +\,  \sum_{i=0}^{k-1} \int \eta^m   \n^i Rm * \n^{k-1-i}\D\phi * \n^k \phi * \phi * \phi \\
& = & -\, \frac{1}{32} \int \eta^m |\langle e_p \cdot \phi,  \nabla^k  \D \phi\rangle|^2  + \frac{1}{32}\int \eta^m \langle \phi, e_p \cdot \n^k \D \phi\rangle \langle \n_p \n^k \phi, \phi \rangle\\
& & +\, \frac{1}{32}\int m \eta^{m-1}\n_q \eta \langle \phi, e_p \cdot  \nabla^k\D \phi\rangle \langle \n^k \phi,  e_q \wedge e_p \cdot \phi \rangle \\
& & + \int \eta^m \n^{k}\D \phi * \n^k \phi * \n \phi * \phi \\
& & +  \sum_{i=0}^{k-1} \int \eta^m   \n^i Rm * \n^{k-1-i}\D\phi * \n^k \phi * \phi * \phi \\
& & + \sum_{i=0}^{k-1} \int \eta^m \n^k \D\phi * \n^i Ric * \n^{k-1-i}\phi * \phi * \phi .
\end{eqnarray*}

Then similarly as before we can derive
\begin{eqnarray*}
II_2
& \leq & \frac{1}{16}\int \eta^m |\n^{k+1} \phi|^2
 + C \int \eta^m |\n^k \phi|^2 + \eta^m |\n^{k-1} Rm|^2 \cr
&     & +\, C \int_{\eta>0} |\n \phi|^2 + |Rm|^2 + |\phi|^2.
\end{eqnarray*}

\begin{eqnarray*}
II_3 & = & \sum_{p=0}^k \sum_{q=0}^{k+1-p} \int \eta^m \langle \n^k \phi, \nabla^{1+p}\phi * \nabla^{q}\phi * \nabla^{k+1-p-q}\phi \rangle \cr
& = & \sum_{p=1}^{k-1} \sum_{q=1}^{k+1-p} \int \eta^m \langle \n^{k-1} \n\phi, \nabla^{p}\n\phi * \nabla^{q-1}\n \phi * \nabla^{k-p-q}\n\phi \rangle \cr
&  & + \sum_{q=1}^k \int \eta^m \langle \nabla^{k-1}\nabla \phi, \nabla \phi * \nabla^{q-1}\nabla \phi * \nabla^{k-q}\nabla \phi\rangle \cr
&  & +  \int \eta^m \langle \n^k \phi, \n^{k+1}\phi * \n \phi * \phi \rangle\cr
&  & +  \sum_{p=1}^{k-1} \int \eta^m \langle \n^k \phi, \nabla^{1+p}\phi * \phi * \nabla^{k+1-p}\phi \rangle\cr
& \leq & \frac{1}{16} \int \eta^m |\n^{k+1} \phi|^2 + C \left( \int \eta^m |\n^k \phi|^2 + \int_{\eta > 0} |\n \phi|^2 + |\phi|^2 \right).
\end{eqnarray*}
The last term
\begin{eqnarray*}
II_5 & = & \sum_{1\leq l \leq k } \int \eta^m \langle \n^k \phi, \nabla^l Ric * \nabla ^{k-l} \phi \rangle \cr
& = & \int m\eta^{m-1}\langle \n \eta \otimes \n^k \phi, \nabla^{k-1} Ric *  \phi \rangle + \eta^m \langle \n^{k+1}\phi, \n^{k-1} Ric * \phi\rangle \cr
&   & + \int \eta^m \langle \n^{k}\phi, \n^{k-1} Ric * \n \phi\rangle+ \sum_{1\leq l \leq k-1 } \int \eta^m \langle \n^k \phi, \nabla^l Ric * \nabla ^{k-l} \phi \rangle \cr
& \leq & \frac{1}{16}\int \eta^m |\n^{k+1} \phi |^2 + C \int \eta^m |\n^k \phi|^2 + C \int_{\eta>0} |\n \phi|^2 + |\phi|^2 \cr
&      & +\, C \int \eta^{m}|\n^{k-1} Ric|^2 + C \int_{\eta>0} |Ric|^2,
\end{eqnarray*}
where we used integral by part in the equation above (hence lost a copy of the cutoff function $\eta$).
Similarly we can estimate
\begin{eqnarray*}
&   & \sum_{0\leq l \leq k-1} \int \eta^m \langle \n^k \phi, \nabla^l Rm * \nabla^{k-l}\phi \rangle \\
& \leq & C \int \eta^m |\n^{k-1} Rm|^2 + \eta^m |\n^k \phi|^2 + C\int_{\eta>0} |Rm|^2 + |\n \phi|^2.
\end{eqnarray*}
\end{proof}

Now we are ready to prove Theorem \ref{l2Estimates}.
\begin{proof}[Proof of Theorem \ref{l2Estimates}]
We can rescale the flow parabolically such that $|Rm|\leq 1$, $|\n^2 \phi|\leq 1$ and $|\n \phi|^2 \leq 1$ on $B_{g(0)}(r) \times [0, a]$, for simplicity we still denote it as $(M, g(t), \phi(t) )$. It is easy to see the volume can be controlled under the flow
\[
e^{-C(n)a}V \leq Vol_{g(t)}(B_{g(0)}(r) ) \leq e^{C(n) a} V,
\]
where $V$ can be taken as $ Vol_{g(s)}(B_{g(0)}(r))$ for any fixed $s\in[0, a]$. Choose a cut-off function $\eta$ supported on $B_{g(0)}(3r/4)$, with $\eta=1$ on $B_{g(0)}(r/2)$ and $|\n \eta|\leq 4/r$.

We will prove the theorem by induction on $k$. When $k=0$ it is trivial. Without loss of generality we can suppose the result hold for $k$ on a larger ball with radius $\frac{3r}{4}$. Take $m= 2k+4$. Define
\begin{equation}\label{eqn: F_i}
F_i (t) = \alpha \int \eta^m |\n^i Rm|^2 dv(t)+ \int \eta^m |\n^{i+2} \phi|^2 dv(t),
\end{equation}
for some constant $\alpha$, $i = 0,1,2,...,k+1$.

By Lemma \ref{lemma: ddt of L2 integral of derivatives of Rm} and \ref{lemma: ddt of L2 integral of derivatives of phi} we can choose $\alpha$ large enough such that
\begin{equation}\label{eqn: derivative of F_i}
\frac{d}{dt} F_k
\leq  - \beta_k F_{k+1} + C_k F_k + C_k V
\end{equation}
for constants $\beta_k$ and $C_k$ depending on $n,a,r$ and $k$. Then let
\[
Q_k =t^{k+1}F_{k+1} + \gamma \sum_{i=0}^k \frac{(k+1)\cdot k\cdot \cdot \cdot (k-i+1)}{\beta_{k}\cdot \beta_{k-1}\cdot \cdot \cdot \beta_{k-i}} t^{k-i} F_{k-i}
\]
where $\gamma$ can be properly chosen such that $k+1+ aC_{k+1}- (k+1)\gamma <0$, hence by induction hypothesis we have
\[
\frac{d}{dt} Q_k \leq C V
\]
for $t\in[0,a]$. We can integrate the above inequality, and rescale the flow to obtain the desired estimates.
\end{proof}


\section{Long time existence}\label{section: long time existence}
\noindent In this section, we will show that the norm of the second order covariant derivative of the spinor field becoming unbounded is the only obstruction of long-time existence of the spinor flow by proving Theorem \ref{blowUpCondition}. { We will assume that the dimension of the manifold $n\geq 3$, since the result on 2-dimensional surfaces has been shown in \cite{Sc2018}.}

\subsection{Convergence of metrics}\label{Convergence of metrics}
Let $(g(t), \phi(t))$, $t\in[0, T)$, for some $T<\infty$, be a solution to the system (\ref{spinorRicciFlow}) on a closed manifold $M^n$ satisfying \begin{equation*}
\sup\limits_{M\times[0, T)}|\nabla^2 \phi|<\infty.
\end{equation*}
Then by the equation (\ref{eqn: Ricci-derivative-of-spinor}), we have
\begin{equation*}
\sup\limits_{M\times[0, T)}|Ric|<\infty.
\end{equation*}

By Lemma \ref{lemma: second-derivative-control-frist} below, we also have
\begin{equation*}
\sup\limits_{M\times[0, T)}|\nabla \phi|<\infty.
\end{equation*}
Hence by Lemma \ref{estimateOfRm} the norm of Riemannian curvature tensor is bounded
\begin{equation*}
\sup\limits_{M\times[0, T)}|Rm|<\infty.
\end{equation*}

Therefore the left hand side of the metric evolving equation in (\ref{spinorRicciFlow}) is uniformly bounded, i.e. we have
\begin{equation*}
\sup\limits_{M\times[0, T)}\left|\frac{\partial}{\partial t}g\right|<\infty.
\end{equation*}
Thus there exists a constant $C<\infty$ such that
\begin{equation}\label{eqn: metrics-equivalence}
e^{-C}g(0)\leq g(t) \leq e^{C}g(0)
\end{equation}
on $M$ for all $t\in[0, T)$. Moreover, as $t\nearrow T$, the metrics $g(t)$ converge uniformly to a continuous metric $g(T)$ such that
\begin{equation*}
e^{-C}g(0)\leq g(T) \leq e^{C}g(0)
\end{equation*}
on $M$.

As direct consequences of the equivalence of metrics in (\ref{eqn: metrics-equivalence}), we have the following uniform Sobolev inequalities. For any given $1\leq p<n$, there exists a constant $C$ independent of $t$, such that
\begin{equation}\label{eqn: Sobolev-inequality1}
\|f\|_{L^{q}(M, g(t))}\leq C\left(\|\nabla f\|_{L^{p}(M, g(t))}+\|f\|_{L^{p}(M, g(t))}\right),
\end{equation}
for any smooth function $f\in C^{\infty}(M\times[0, T))$, $q\leq \frac{np}{n-p}$, and all $t\in[0, T)$.

Furthermore, for any given $p\geq 1$ satisfying $1-\frac{n}{p}>0$, there exists a constant $C$ independent of $t$, such that
\begin{equation}\label{eqn: Sobolev-inequality2}
\|f\|_{C^{0}(M)}(t)\leq C\left(\|\nabla f\|_{L^{p}(M, g(t))}+\|f\|_{L^{p}(M, g(t))}\right),
\end{equation}
for any smooth function $f\in C^{\infty}(M\times[0, T))$ and all $t\in[0, T)$.

From the $L^2$-estimates in Theorem \ref{l2Estimates}, by using the uniform Sobolev inequality in (\ref{eqn: Sobolev-inequality1}) certain times with $f:=|\nabla^k Rm|$ and $f:=|\nabla^{2+k} \phi|$, and also by Kato's inequality, we can obtain
\begin{equation*}
\sup\limits_{t\in[0, T)}\|\nabla^k Rm\|_{L^{p}(M, g(t))}\leq C,
\end{equation*}
\begin{equation*}
\sup\limits_{t\in[0, T)}\|\nabla^{2+k} \phi\|_{L^{p}(M, g(t))}\leq C,
\end{equation*}
for certain $p$ satisfying $1-\frac{n}{p}>0$. Then the following pointwise estimates follows from the Sobolev inequality in (\ref{eqn: Sobolev-inequality2}).
\begin{lem}\label{lemma: pointwise-estimate}
Let $(g(t), \phi(t))$, $t\in[0, T)$, for some $T<\infty$, be a solution to the system (\ref{spinorRicciFlow}) satisfying $\sup\limits_{M\times[0, T)}|\nabla^2\phi|<\infty$. Then for each non-negative integer $k$, there exists a constant $C_k$, such that
\begin{equation}\label{pointwise-curvature-estimate}
\sup_{M\times[0, T)}|\nabla^k Rm|\leq C_k,
\end{equation}
\begin{equation}\label{pointwise-spinor-estimate}
\sup_{M\times[0, T)}|\nabla^{2+k}\phi|\leq C_k.
\end{equation}
\end{lem}

From these pointwise estimates, we can obtain the following metrics convergence result, by using essentially the same argument as showing blow up condition for Ricci flow (see details in e.g. \S 7 in Chapter 6 in \cite{CK2004}), except replacing $-2Ric$ by $-\frac{1}{8}R_{jk} - \frac{1}{4}|\nabla \phi|^2 g_{jk}+  \frac{1}{4}\langle \D \phi, e_j \cdot \nabla_{e_k}\phi + e_k \cdot \nabla_{e_j}\phi \rangle$, which and whose derivatives are uniformly bounded by Lemma \ref{lemma: pointwise-estimate}.

\begin{prop}\label{prop: metric-convergence}
Let $(g(t), \phi(t))$, $t\in[0, T)$, for some $T<\infty$, be a solution to the system (\ref{spinorRicciFlow}) satisfying $\sup\limits_{M\times[0, T)}|\nabla^2\phi|<\infty$. Then there exists a smooth metric $g(T)$ on $M^n$ such that $g(t)\rightarrow g(T)$ in any $C^{k}$ norm as $t\nearrow T$.
\end{prop}

\subsection{Convergence of spinors}

Now in order to extend the solution $(g(t), \phi(t))$, $t\in[0, T)$, crossing the time $t=T$, we also need to choose a spinor $\phi(T)\in \Gamma(\Sigma_{g(T)}M)$, which will be the limit of $\phi(t)$ in certain sense as $t\nearrow T$.

Let $P_{SO}(t)$, $t\in[0, T]$ denote orthonormal frame bundles with respect to the metric $g(t)$. For each $t\in[0, T]$, there exists a unique $A(t)\in \End(TM)$ that gives a principal $\SO(n)$-bundle isomorphism
\begin{eqnarray*}
A(t): P_{SO}(t) & \longrightarrow & P_{SO}(T) \\
e=(e_1, \cdots, e_n) & \longmapsto &  A(t)e=(A(t)e_1, \cdots, A(t)e_n).
\end{eqnarray*}
In particular, $A(T)=id_{TM}$.

Let $\xi_t: P_{Spin}(t)\rightarrow P_{SO}(t)$, $t\in[0, T]$, be {\em metric} spin structures with respect to the metric $g(t)$ corresponding to a fixed {\em topology} spin structure on $M$. Then the isomorphism $A(t)$ can be lifted to be an isomorphism $\tilde{A}(t): P_{Spin}(t)\rightarrow P_{Spin}(T)$, for each $t\in [0, T)$, and such that there is the following commutative diagram
$$
\xymatrix{
P_{Spin}(t) \ar[d]_{\xi_t} \ar[r]^{\tilde{A}(t)} & P_{Spin}(T)\ar[d]^{\xi_T}\\
P_{SO}(t) \ar[r]^{A(t)} & P_{SO}(T).}
$$

Furthermore, $\tilde{A}(t)$ induce isometries, which are still denoted by $A(t)$, between spinor bundles given by
\begin{eqnarray*}
A(t): \Sigma_{t}M:=P_{Spin}(t)\times_{\mu}\Sigma_n & \longrightarrow &  \Sigma_{T}M:=P_{Spin}(T)\times_{\mu}\Sigma_n \\
\phi=[\tilde{e}, \tilde{\phi}] & \longmapsto &  A(t)\phi=[\tilde{A}(t)\tilde{e}, \tilde{\phi}].
\end{eqnarray*}
As shown in Proposition 2 in \cite{BG1992}, for any fixed point $x\in M$, the isometries $A(t)$ coincide with the fiber-wise parallel transport from $\Sigma_{t, x}M$ to $\Sigma_{T, x}M$ along the path $g_s=(1-s)g(t)+sg(T)$, $0\leq s\leq1$, associated with the Bourguignon-Gauduchon connection for the universal spinor bundle.

Let $\nabla^t$ denote the Levi-Civita connection with respect to the metric $g(t)$ for all $t\in [0, T]$. Define a connection
\begin{equation*}
\overline{\nabla}^{t}:=A(t)^{-1}\circ\nabla^{T}\circ A(t)
\end{equation*}
on the tangent bundle for each $t\in [0, T]$.
This connection is compatible with the metric $g(t)$ and has a torsion, for any $t\in[0, T)$, given  by
\begin{equation*}
\begin{aligned}
\overline{T}^{t}(X, Y)
&:=\overline{\nabla}^{t}_{X}Y-\overline{\nabla}^{t}_{Y}X-[X, Y]\\
&=(\nabla^{T}_{Y}A(t)^{-1})A(t)X-(\nabla^{T}_XA(t)^{-1})A(t)Y.
\end{aligned}
\end{equation*}
Then the difference between $\nabla^{t}$ and $\overline{\nabla}^{t}$ is given in terms of the torsion as
\begin{equation}\label{eqn: connection-difference}
\begin{aligned}
&2g(t)( \overline{\nabla}^{t}_{X}Y-\nabla^{t}_{X}Y, Z)
=g(t)(\overline{T}^{t}(X, Y), Z)\\
&\qquad \qquad \qquad -g(t)(\overline{T}^{t}(X, Z), Y) -g(t)(\overline{T}^{t}(Y, Z), X).
\end{aligned}
\end{equation}

Let $\{e_{1}(t), \cdots, e_{n}(t)\}$ be a local orthonormal frame with respect to $g(t)$, and $\omega(t)$ and $\overline{\omega}(t)$ be the connection 1-forms for $\nabla^{t}$, whose components are
\begin{equation}\label{eqn: connection-1-forms}
\begin{aligned}
\omega(t)_{ij} & =g(t)(\nabla^{t}e_{i}(t), e_{j}(t)),\\
\overline{\omega}(t)_{ij} &=g(t)(\overline{\nabla}^{t}e_{i}(t), e_{j}(t)).
\end{aligned}
\end{equation}

Metric connections $\nabla^{t}$ and $\overline{\nabla}^{t}$ naturally induce connections on the spinor bundle $\Sigma_{t}M$ associated to the metric $g(t)$, which are still denoted by $\nabla^{t}$ and $\overline{\nabla}^{t}$, respectively. Then by the formula for the induced spinor connections, we have
\begin{equation*}
\overline{\nabla}^{t}_{X}\phi(t)-\nabla^{t}_{X}\phi(t)=\frac{1}{4}\sum^{n}_{i,j=1}(\overline{\omega}^{t}_{ij}-\omega^{t}_{ij})(X)e_{i}\cdot e_{j}\cdot\phi(t).
\end{equation*}
Then combining with equations (\ref{eqn: connection-difference}) and $(\ref{eqn: connection-1-forms})$, we have
\begin{equation}\label{eqn: spinor-connection-difference}
(\overline{\nabla}^{t}-\nabla^{t})\phi(t)=\nabla^{T}(A(t)^{-1}) * A(t) * \phi(t).
\end{equation}

\begin{lem}\label{lemma: spinor-l2-estimate}
Let $(g(t), \phi(t))$, $t\in[0, T)$, for some $T<\infty$, be a solution to the system (\ref{spinorRicciFlow}) satisfying $\sup\limits_{M\times[0, T)}|\nabla^2\phi|<\infty$, and $g(T)$ be the limit smooth metric obtained in Proposition \ref{prop: metric-convergence}. Then for each nonnegative integer $k$, there exists a constant $C_{k}$ independent of $t$ such that
\begin{equation*}
\|(\nabla^{T})^{k}(A(t)\phi(t))\|_{L^{2}(\Sigma_{T}M, g(T))}\leq C_{k}.
\end{equation*}
\end{lem}
\begin{proof}
\begin{eqnarray*}
& &\quad \|(\nabla^{T})^{k}(A(t)\phi(t))\|^{2}_{L^{2}(\Sigma_{T}M, g(T))}\\
& =& \int_{M}|(\nabla^{T})^{k}(A(t)\phi(t))|^2_{g(T)}dv(T)\\
& \leq &  C \int_{M}|A(t)^{-1}(\nabla^{T})^{k}(A(t)\phi(t))|^2_{g(t)}dv(t)\\
& = & C \int_{M}|(\overline{\nabla}^{t})^{k}\phi(t)|^2_{g(t)}dv(t)\\
& = & C \int_{M}|(\nabla^{t}+\overline{\nabla}^{t}-\nabla^{t})^{k}\phi(t)|^2_{g(t)}dv(t)\\
& \leq & C \sum^{k}_{i=0}\int_{M}|(\nabla^{t})^{i}\phi(t)|^2_{g(t)}dv(t)\\
& \leq & C_{k}^2.
\end{eqnarray*}
In above steps the constants C may vary along steps but all of them are independent of $t$. The first inequality follows from the uniform equivalence of metrics $g(t)$ and $g(T)$ for all $t\in[0, T]$. The second inequality follows from the expression (\ref{eqn: spinor-connection-difference}) and the fact $|(\nabla^{t})^{m}\nabla^{T}A(t)^{-1}|$ and $|(\nabla^{t})^{m}A(t)|$ are uniformly bounded for all $t\in[0, T]$, since essentially $A(t)=(g(T)^{-1}g(t))^{\frac{1}{2}}$ and all partial derivatives of $g(t)$ are uniformly bounded (this shall be shown while proving the convergence of metrics). Finally, the last inequality follows from $L^{2}$ estimates in Theorem \ref{l2Estimates}.
\end{proof}

From the $L^{2}$ estimates in Lemma \ref{lemma: spinor-l2-estimate} and standard Sobolev inequality with respect to metric $g(T)$, one immediately obtains uniform pointwise estimates. Furthermore, we have
\begin{prop}\label{prop: spinor-convergence}
Let $(g(t), \phi(t))$, $t\in[0, T)$, for some $T<\infty$, be a solution to the system (\ref{spinorRicciFlow}) satisfying $\sup\limits_{M\times[0, T)}|\nabla^2\phi|<\infty$, and $g(T)$ be the limit smooth metric obtained in Proposition \ref{prop: metric-convergence}. Then there exists a spinor $\phi(T)\in\Gamma(\Sigma_{T}M)$ such that $A(t)\phi(t)\rightarrow \phi(T)$ in any $C^{k}$ norm as $t\nearrow T$.
\end{prop}

Now one can solve the system (\ref{spinorRicciFlow}) with the initial data $(g(T), \phi(T))$ obtained in Propositions \ref{prop: metric-convergence} and \ref{prop: spinor-convergence} to extend the flow crossing the time $t=T$. This completes the proof of Theorem \ref{blowUpCondition}.


\section{Lower bound estimate for the existence time}\label{section: lower bound estimate for the existence time}
\noindent First we show the easy fact that $|\n \phi|$ is naturally controlled by $|\n^2 \phi|$ under the normalization $|\phi| \equiv 1$. Hence to obtain a lower bound for the existence time, we only need to control $|\n^2 \phi|$.

\begin{lem}\label{lemma: second-derivative-control-frist}
Let $\phi$ be a spinor field with $|\phi| \equiv 1$ on a domain $\Omega$, then
\[
|\n \phi|_{\infty, \Omega}^2 \leq |\Delta \phi|_{\infty, \Omega}.
\]
\end{lem}
\begin{proof}
Let $0\leq \eta \leq 1$ by any cut-off function supported on $\Omega$ with $\eta=1$ on $\Omega_{-\epsilon}{ :=\{x\in \Omega \mid d(x, \partial \Omega)> \epsilon\}}$. Since
\[\langle \n \phi, \phi\rangle = 0,\]
we use integral by part to get
\begin{eqnarray*}
\int \eta |\n \phi|^{2k}
& = &\int \eta |\n \phi|^{2k-2} \langle \phi, \Delta \phi\rangle \cr
& \leq  & |\Delta \phi|_{\infty, r} \int \eta |\n \phi|^{2k-2}\cr
& \leq & \cdots \cr
& \leq & |\Delta \phi|_{\infty, r} ^k
\end{eqnarray*}
for any positive interger $k$. Take the k-th root and let $k\to \infty$ yields
\[ |\n \phi|_{\infty, \Omega_{-\epsilon}}^2 \leq |\Delta \phi|_{\infty, \Omega}.\]
Let $\epsilon \to 0$ finishes the proof.
\end{proof}

\begin{proof}[proof of Theorem \ref{lowerBoundForExistenceTime}]
Without loss of generality we can assume $K=1$. For $\Lambda > 2$ to be determined later, suppose $[0,T]$ ($T\leq 1$) is the maximal time interval such that
\[
\quad \sup_{M \times [0, T]}|\n^2 \phi| \leq \Lambda  .
\]
By Lemma \ref{estimateOfRm} we have $|Rm| \leq C(n,L,\Lambda)$ on $M\times[0,T]$ (note that we can take $\eta\equiv 1$ in its proof since $M$ is closed).

Let $F_i(t)$ be the quantities defined in (\ref{eqn: F_i}). Let $k=\left[\frac{n}{2}\right]+1$, and set
\[
P_k(t) = F_k(t) + \sum_{i=1}^{k} \frac{C_k C_{k-1}...C_{k+1-i}}{\beta_{k-1}\beta_{k-2}...\beta_{k-i}}F_{k-i}(t),
\]
\[
P_{k-1}(t) = F_{k-1}(t) + \sum_{i=1}^{k-1} \frac{C_{k-1}...C_{k-i}}{\beta_{k-2}...\beta_{k-1-i}}F_{k-1-i}(t),
\]
where the constants are from (\ref{eqn: derivative of F_i}). By  (\ref{eqn: derivative of F_i}) we have
\[
P_k(t) \leq P_k(0) + C(n, L, \Lambda)V t,
\]
and
\[
P_{k-1}(t) \leq P_{k-1}(0) + C(n, L, \Lambda)V t.
\]
Recall here $V=\Vol_{g(0)}M$. In particular,
\[
\int|\n^{k+2}\phi|^{2}dv(t) \leq P_{k}(0) + C(n, L, \Lambda)V t,
\]
and
\[
\int|\n^{k+1}\phi|^{2}dv(t) \leq P_{k-1}(0) + C(n, L, \Lambda)V t.
\]

Moreover, by setting $k=0$ in (\ref{eqn: derivative of F_i}), dropping the term $-\beta_{0}F_{1}$ on the right hand side of the inequality, and then solving the differential inequality, one can easily get
\[
F_{0}(t)\leq F_{0}(0) + (e^{C_{0}t}-1)V.
\]
In particular,
\[
\int|\n^{2}\phi|^{2}dv(t) \leq F_{0}(0) + (e^{C_{0}t}-1)V.
\]

Since the Sobolev constants are uniform along the flow up to a factor $e^{C(n, L, \Lambda)t}$, we can use Lemmas \ref{lemma: lp interpolation} and \ref{lemma: multiplicative sobolev inequality} to get pointwise estimate for $|\n^{2}\phi|$ as follows. In the rest of the proof, we use $C$ to denote constants only depending on $n$, $C_{1}$ only depending on $n$ and the initial metric $g(0)$, and $C_{2}$ also depending on $\Lambda$. In the following derivations, constants $C$, $C_{1}$, and $C_{2}$ may vary along steps.

By setting $i=1$, $\eta\equiv1$, and $A=\n \phi$ in Lemma \ref{lemma: lp interpolation}, we have
\begin{eqnarray*}
&      & \|\n^2\phi\|_{L^{2k}(g(t))}\cr
& \leq & C\left(|\n \phi|_{\infty}(t)\right)^{1-\frac{1}{k}}\left(\|\n^{k+1}\phi\|_{L^{2}(g(t))} + \|\n \phi\|_{L^{2}(g(t))}\right)^{\frac{1}{k}}\cr
& \leq & C\left(|\n^{2} \phi|_{\infty}(t)\right)^{1-\frac{1}{k}}\left(\|\n^{k+1}\phi\|_{L^{2}(g(t))} + \left(\int|\Delta\phi|dv(t)\right)^{\frac{1}{2}}\right)^{\frac{1}{k}}\cr
& \leq & C\left(|\n^{2} \phi|_{\infty}(t)\right)^{1-\frac{1}{k}}\left(\|\n^{k+1}\phi\|_{L^{2}(g(t))} + \left(\int\sqrt{n}|\n^{2}\phi|dv(t)\right)^{\frac{1}{2}}\right)^{\frac{1}{k}}\cr
& \leq & C\left(|\n^{2} \phi|_{\infty}(t)\right)^{1-\frac{1}{k}}\left(\|\n^{k+1}\phi\|_{L^{2}(g(t))} + \sqrt{nV}e^{C_{2}t}+\|\n^{2}\phi\|_{L^{2}(g(t))}\right)^{\frac{1}{k}}.\cr
\end{eqnarray*}
Similarly, by setting $A=\n^{2}\phi$ in Lemma \ref{lemma: lp interpolation}, we have
\begin{equation*}
\|\n^{3}\phi\|_{L^{2k}(g(t))}\leq C \left(|\n^{2}\phi|_{\infty}(t)\right)^{1-\frac{1}{k}}\left(\|\n^{k+2}\phi\|_{L^{2}(g(t))} + \|\n^{2}\phi\|_{L^{2}(g(t))}\right)^{\frac{1}{k}}
\end{equation*}

By setting $p=2k$ and $u = |\n^{2}\phi|$ in Lemma \ref{lemma: multiplicative sobolev inequality}, and combining with the above estimates, we have
\begin{eqnarray*}
&      & \left(|\n^{2}\phi|_{\infty}(t)\right)^{\frac{2k}{\alpha}}\cr
& \leq & C_{1}e^{C_{2}t}(C_{S}(g(t)))^{2k}\left(\|\n^{2}\phi\|_{L^{m}(g(t))}\right)^{\frac{2k}{\alpha}-2k}\Big(\|\n^{3}\phi\|^{2k}_{L^{2k}(g(t))}\cr
&      & +\, C_{1}e^{C_{2}t}\|\n^{2}\phi\|^{2k}_{L^{2k}(g(t))}\Big)\cr
& \leq &  C_{1}e^{C_{2}t}(1+e^{C_{2}t})(C_{S}(g(t)))^{2k}V^{\frac{1}{m}\left(\frac{2k}{\alpha}-2k\right)}
\left(|\n^{2}\phi|_{\infty}(t)\right)^{\frac{2k}{\alpha}-2}\Big( \int|\n^{k+2}\phi|^{2}dv(t)\cr
&      & +\int|\n^{k+1}\phi|^{2}dv(t) + \int|\n^{2}\phi|^{2}dv(t) + nVe^{C_{2}t}\Big)\cr
& \leq &  C_{1}e^{C_{2}t}(1+e^{C_{2}t})(C_{S}(g(t)))^{2k}V^{\frac{1}{m}\left(\frac{2k}{\alpha}-2k\right)}
\left(|\n^{2}\phi|_{\infty}(t)\right)^{\frac{2k}{\alpha}-2}( P_{k}(0) + P_{k-1}(0)\cr
&      & +\, F_{0}(0) + C_{2}Vt + (e^{C_{0}t}-1)V + nVe^{C_{2}t}).
\end{eqnarray*}
Then dividing by $\left(|\n^{2}\phi|_{\infty}(t)\right)^{\frac{2k}{\alpha}-2}$, we have
\begin{eqnarray*}
(|\n^{2}\phi|_{\infty}(t))^{2}
& \leq & C_{1}e^{C_{2}t}(1+e^{C_{2}t})(C_{S}(g(t)))^{2k}V^{\frac{1}{m}\left(\frac{2k}{\alpha}-2k\right)}( P_{k}(0) + P_{k-1}(0)\\
&      & +\, F_{0}(0) + C_{2}Vt + (e^{C_{0}t}-1)V + nVe^{C_{2}t}).
\end{eqnarray*}
Recall that $C_{S}(g(0))=C_{1}V^{-\frac{1}{n}}$, Sobolev constants $C_{S}(g(t))$ are uniform along the flow up to a factor $e^{C_{2}t}$, and $\frac{1}{\alpha}=\left(\frac{1}{n}-\frac{1}{2k}\right)m+1$, thus
\begin{eqnarray*}
(|\n^{2}\phi|_{\infty}(t))^{2}
& \leq & C_{1}e^{C_{2}t}(1+e^{C_{2}t})V^{\frac{1}{m}\left(\frac{2k}{\alpha}-2k\right)-\frac{2k}{n}}( P_{k}(0) + P_{k-1}(0)\\
&      & +\, F_{0}(0) + C_{2}Vt + (e^{C_{0}t}-1)V + nVe^{C_{2}t})\\
& \leq & C_{1}e^{C_{2}t}(1+e^{C_{2}t})V^{-1}( P_{k}(0) + P_{k-1}(0)\\
&      & +\, F_{0}(0) + C_{2}Vt + (e^{C_{0}t}-1)V + nVe^{C_{2}t})\\
& \leq & C_{1}e^{C_{2}t}(1+e^{C_{2}t})( V^{-1}P_{k}(0) + V^{-1}P_{k-1}(0)\\
&      & +\, V^{-1}F_{0}(0) + C_{2}t + (e^{C_{0}t}-1) + ne^{C_{2}t})
\end{eqnarray*}

Now choose a sufficiently large $\Lambda$ such that
\begin{equation*}
\Lambda^{2}> \max\{4 C_{1}( V^{-1}P_{k}(0) + V^{-1}P_{k-1}(0)+ V^{-1}F_{0}(0) + n ), 4\}.
\end{equation*}
Then we can take $\delta$ such that
\begin{eqnarray*}
\Lambda^{2}
& = & C_{1}e^{C_{2}\delta}(1+e^{C_{2}\delta})( V^{-1}P_{k}(0) + V^{-1}P_{k-1}(0) + V^{-1}F_{0}(0)\\
&   & +\, C_{2}\delta + (e^{C_{0}\delta}-1) + ne^{C_{2}\delta}).
\end{eqnarray*}
This complete the proof.
\end{proof}

\section{Appendix: Interpolation lemma}\label{section: appendix}
\noindent The following interpolation lemma was proved in the Appendix of \cite{KS2002}.
\begin{lem}\label{lemma: lp interpolation}
For integers $k>0$, $1\leq i\leq k$ and $m\geq 2k$ we have the inequality
\[
\left( \int \eta^m |\n^i A|^{\frac{2k}{i}} \right)^\frac{i}{2k} \leq C |A|_\infty^{1-\frac{i}{k}} \left( \left(\int \eta^m |\n^k A|^2\right)^\frac{1}{2} + |A|_{L^2, \eta>0}\right)^{\frac{i}{k}},
\]
where the constant $C$ depends only on $n, m, k$ and $|\n \eta|_\infty$.
\end{lem}
As a corollary we have the following lemma which is needed in our derivation of $L^2$ derivative estimates.
\begin{lem}\label{lemma: l2 interpolation}
Let $1 \leq i_1, i_2, ..., i_l\leq k$ and $i_1+ i_2 + ... + i_l = 2k$, $m \geq 2k$. Then for $($vector valued$)$ tensors $A_1, A_2, ..., A_r$ and a cutoff function $\eta$, we have
\begin{eqnarray*}
&    & \int \eta^m \n^{i_1}A_1 * \n^{i_2}A_2 * ... * \n^{i_l}A_l * A_{l+1}* ... * A_r \\
& \leq & C \left( \prod_{i=l+1}^r |A_i|_\infty \right) \sum_{j=1}^l |A_j|_{\infty}^{\frac{2k}{i_j} - 2}\left( \int \eta^m |\n^k A_j|^2 + \int_{\eta>0} |A_j|^2\right)
\end{eqnarray*}
where the constant $C$ depends only on $n, m, k, r $ and $|\n \eta|_\infty$.
\end{lem}
\begin{proof}
\begin{eqnarray*}
&    & \int \eta^m \n^{i_1}A_1 * \n^{i_2}A_2 * ... * \n^{i_l}A_l * A_{l+1}* ... * A_r \cr
& \leq & C \left( \prod_{i=l+1}^r |A_i|_\infty \right) \prod_{j=1}^l \left( \int \eta^m |\n^{i_j} A_j|^{\frac{2k}{i_j}} \right)^\frac{i_j}{2k} \cr
& \leq & C \left( \prod_{i=l+1}^r |A_i|_\infty \right) \prod_{j=1}^l |A_j|^{1-\frac{i_j}{k}}\left( \int \eta^m |\n^k A_j|^2 + \int_{\eta>0} |A_j|^2 \right)^{\frac{i_j}{2k}} \cr
& \leq & C \left( \prod_{i=l+1}^r |A_i|_\infty \right) \sum_{j=1}^l |A_j|_{\infty}^{\frac{2k}{i_j} - 2}\left( \int \eta^m |\n^k A_j|^2 + \int_{\eta>0} |A_j|^2\right)
\end{eqnarray*}
where in the last step we used Young\rq{}s inequality.
\end{proof}

Recall that (\cite{Sa1992}) on any complete Riemannian manifold with $Ric \geq -K$ we have a Sobolev inequality
\begin{equation}\label{Sobolev inequality}
\left( \int_{B(r)} u^{\frac{2\mu}{\mu-2}}\right)^\frac{\mu-2}{\mu} \leq C(n) \frac{r^2 e^{\sqrt{K}r}}{V(r)^\frac{2}{\mu}} \int_{B(r)} (|\n u |^2 + r^{-2} u^2),
\end{equation}
for any $C^1$ function compactly supported on a geodesic ball $B(r)$, where $\mu=n$ when $n\geq 3$ and $2<\mu< \infty$ when $n=2$.  Using this Sobolev inequality we can prove a multiplicative version by the same argument as in \cite{KS2002} (Theorem 5.6).
\begin{lem}\label{lemma: multiplicative sobolev inequality}
Let $u \in C_c^1(B(r))$, where $B(r)$ is a geodesic ball with radius $r$ on a $n$-dimensional Riemannian manifold with $Ric \geq -K$. For any $p > n$, $m\geq 0$ and $0 < \alpha \leq 1$ with $\frac{1}{\alpha} = (\frac{1}{n} - \frac{1}{p})m +1$, there is a constant $C$ depending only on $n, m,p$ such that
\[
|u|_\infty \leq C C_S^\alpha |u|_{L^m}^{1-\alpha}(|\n u|_{L^p} + { r^{-1}}|u|_{L^p})^ \alpha,
\]
where $C_S =C(n) \frac{{ r e^{\sqrt{K}r/2}}}{V(r)^\frac{1}{n}}$.
\end{lem}

\section{Acknowledgements} C. Wang gratefully acknowledges the support and wonderful working
condition of the Max Planck Institute for Mathematics in Bonn. Both authors would like to thank Professor McKenzie Wang for his interests on this note and his encouragement.

\end{document}